\documentclass[a4paper]{amsart}
\usepackage{a4wide}
\usepackage[english]{babel}
\usepackage[utf8]{inputenc}
\usepackage{amsmath,amssymb,amsthm}
\usepackage[table]{xcolor}
\usepackage{graphicx,enumerate,url}
\usepackage[colorinlistoftodos]{todonotes}
\usepackage{footmisc}

\usepackage{tikz}
\usetikzlibrary{calc,through,backgrounds,shapes,matrix,decorations.markings,decorations.pathmorphing,arrows.meta}
\usepackage{caption}

\usepackage{array}
\usepackage{multirow}
\usepackage{multicol}
\usepackage{hhline}

\usepackage{hyperref}
\usepackage{enumitem}
\setenumerate{label*=\arabic*., itemsep=5pt, topsep=5pt}

\usepackage{url, hypcap}
\hypersetup{colorlinks=true, citecolor=darkblue, linkcolor=darkblue}

\numberwithin{figure}{section}
\numberwithin{equation}{section}

\usepackage[
backend=biber,
style=numeric]{biblatex}
\usepackage{csquotes}

\usepackage{cleveref}

\newtheorem{theorem}{Theorem}[section]
\newtheorem{proposition}[theorem]{Proposition}
\newtheorem{lemma}[theorem]{Lemma}
\newtheorem{corollary}[theorem]{Corollary}

\theoremstyle{definition}
\newtheorem{example}[theorem]{Example}
\newtheorem{remark}[theorem]{Remark}
\newtheorem{question}[theorem]{Question}

\newcommand{\ZZ}{\mathbb{Z}}
\newcommand{\NN}{\mathbb{N}}
\newcommand{\RR}{\mathbb{R}}

\newcommand{\set}[2]{\left\{ #1 \;|\; #2 \right\}}
\newcommand{\bigset}[2]{\left\{ #1 \;\big|\; #2 \right\}}

\newcommand{\Lin}{\mathsf{Lin}}

\newcommand{\Sym}{{\mathfrak{S}}}
\newcommand{\Perm}{\operatorname{Sym}}
\newcommand{\des}{\operatorname{des}}
\newcommand{\Des}{\operatorname{Des}}
\newcommand{\asc}{\operatorname{asc}}
\newcommand{\Asc}{\operatorname{Asc}}
\newcommand{\id}{\operatorname{id}}

\newcommand{\Num}{\mathcal{N}}
\newcommand{\Poin}{\operatorname{Poin}}
\newcommand{\Euler}{\operatorname{E}}
\newcommand{\ino}{\operatorname{ino}}   
\newcommand{\Ino}{\operatorname{Ino}}
\newcommand{\stat}{\operatorname{st}}

\newcommand{\hb}[1]{\overline{#1}}  

\newcommand{\rk}{\operatorname{rk}}
\newcommand{\regions}{\mathcal{R}}

\newcommand{\Dfn}[1]{\emph{\bfseries #1}}

\definecolor{darkblue}{rgb}{0,0,0.7}

\definecolor{lightblue}{rgb}{0.68,0.85,1}

\definecolor{lightgrey}{rgb}{0.9,0.9,0.9}

\definecolor{grey}{rgb}{0.5,0.5,0.5}

\definecolor{darkgreen}{RGB}{0,128,0}

\newcommand{\red}{\color{red}}

\newcommand{\orange}{\color{orange}}

\newcommand{\violet}{\color{violet}}

\def\cA{\mathcal{A}}
\def\cB{\mathcal{B}}
\def\cC{\mathcal{C}}

\def\cL{\mathcal{L}}
\def\cM{\mathcal{M}}

\def\cR{\mathcal{R}}

\title[The coarse flag Hilbert-Poincaré series of the braid arrangement]{The coarse flag Hilbert-Poincaré series\\ of the braid arrangement}

\author[E.~Hoster]{Elena Hoster}
\address[E.~Hoster \& C.~Stump]{Fakultät für Mathematik, Ruhr-Universität Bochum, Germany}
\email{\{elena.hoster,christian.stump\}@rub.de}
\author[C.~Stump]{Christian Stump}

\begin{document}

\begin{abstract}
  The paper concerns the coarse flag Hilbert-Poincaré series of \citeauthor{MaglioneVoll} in the case of the braid arrangement~$\cB_n$ associated to the symmetric group~$\Sym_{n+1}$.
  We explicitly construct a companion statistic $\ino : \Sym_{n+1} \times \Perm(n) \rightarrow \NN$ for the descent statistic on $\Perm(n)$ using reverse $(P,\omega)$-partitions and quasisymmetric functions.
\end{abstract}

\thanks{The authors thank Galen Dorpalen-Barry, Theo Douvropoulos, and Raman Sanyal for helpful discussions.}

\maketitle

\section{Introduction and main results}
\label{sec: intro}

\subsection{Background}

Let~$\cA$ be a central and essential hyperplane arrangement in~$V \cong \RR^n$.
Its \Dfn{lattice of flats}~$\cL$ is given by the intersections of hyperplanes in~$\cA$ ordered by reverse inclusion, so that its unique bottom element is~$V$ and its unique top element is~$\mathbf 0$.
The \Dfn{chain Poincaré series} of a chain~$\cC = \{ C_1 < C_2 < \dots < C_k\}$ in~$\cL$ is given by the product of the Poincaré series of the intervals, with bottom and top attached,
\[
  \Poin_\cC(y) = \prod_{i=0}^k \Poin([C_i,C_{i+1}];y)\,,
\]
where we set $C_0 = V$ and $C_{k+1} = \mathbf{0}$.
The \Dfn{coarse flag Hilbert-Poincaré series} is defined by Maglione--Voll in~\cite{MaglioneVoll} as
\[
    \operatorname{cfHP}_\cA (y,t) =
    \sum_{\cC}
    \Poin_\cC(y) \cdot \left( \frac{t}{1-t} \right)^{\#\cC}
    = \frac{\Num_\cA(y,t)}{(1-t)^n} \,,
\]
where the sum ranges over all chains in the intersection poset~$\cL$ not containing the bottom element~$V$.
Its numerator polynomial has nonnegative coefficients~\cite{poincareextended} and the Chow polynomial of the associated matroid is given by the evaluation $\operatorname{cfHP}_\cA (-x,x)$~\cite{stump2024chowaugmentedchowpolynomials}.
It is moreover shown in~\cite{KühneMaglione} for simplicial arrangements that
\[
  \Num_\cA(1,t) = \#\regions \cdot \Euler_n(t) \,,
\]
where $\Euler_n(t) = \sum_{\sigma \in \Perm(n)} t^{\des(\sigma)}$ is the $n$-th \Dfn{Eulerian polynomial}, and where~$\regions$ denotes the set of regions of the arrangement~$\cA$ given by the connected components of its complement.
This raises the following natural question.

\begin{question}
\label{qu:companionstatistic}
  Find a \emph{companion statistic} $\stat : \regions \times \Perm(n) \rightarrow \NN$ for the descent number on $\Perm(n)$ such that
  \[
    \Num_\cA(y,t) = \sum_{\substack{R \in \regions \\ \sigma \in \Perm(n)}} y^{\stat(R,\sigma)}\cdot t^{\des(\sigma)}\,.
  \]
\end{question}

If~$\cA$ is the reflection arrangement of a real reflection group~$W$, the regions are in bijection with the elements of the group~$W$, so this becomes the question of finding a companion statistic on~${W \times \Perm(n)}$.
We emphasize that it already appears to be a challenging task to \emph{find such a companion statistic}, and that this statistic would in particular provide a new approach towards generalizing left-to-right minima on permutations to real reflection groups, see \Cref{rem:generaltype}.

\medskip

The purpose of this paper is to introduce a new statistic on permutations which answers \Cref{qu:companionstatistic} for the symmetric group.

\medskip

Let $\cB_n$ be the \Dfn{braid arrangement} given by the hyperplanes $H_{ij} = \{ x_i = x_j\}$ for $1 \leq i < j \leq n+1 $ in $V = \big\{ (x_1,\dots,x_{n+1}) \in \RR^{n+1} \mid x_1+\dots+x_{n+1} = 0\big\} \cong \RR^n$, the corresponding reflection group is~$W = \Sym_{n+1}$.
Its lattice of flats, denoted by~$\Pi_n$, is given by the \Dfn{lattice of set partitions} of~$\{1,\dots,n+1\}$ of rank~$n$, with cover relations being given by merging two blocks.
Its minimal element is the all-singletons set partition and its maximal element is the all-in-one-block set partition.
A \Dfn{set composition} is a set partition together with a linear ordering of the blocks.
We write~$B_1 \vert B_2 \vert \cdots \vert B_k$ for the set composition with the blocks being ordered from left to right.
The face poset of the braid arrangement~$\cB_n$, denoted by~$\Sigma_n$, is given by the \Dfn{poset of set compositions} of~$\{1,\dots, n+1\}$ of rank~$n$, with cover relations being given by merging two \emph{consecutive} blocks.
Its minimal elements are thus in natural bijection with permutations, and its unique maximal element is the all-in-one-block set composition.

\begin{example}
\label{ex:coarseflagHP}
  The braid arrangement $\cB_2$ is given below with its regions labeled by $w\in \Sym_{3}$ in one-line-notation.
  \begin{center}
    \begin{tikzpicture}[scale=0.75, baseline=-0.5,font=\small ]
    \def\rows{6}
    \def\hrows{\rows/2}
    \def\cut{0.4}
    \def\reglabel{2.5cm} 

    \tikzstyle{barysd} = [color = gray,dashed] 
    \tikzstyle{hyperplanes} = [thick]

    \begin{scope}
        \clip (0,0) circle (\hrows+\cut);

        \draw[hyperplanes]
        (-\rows,0) -- (\rows,0)
        [rotate=60] (-\rows,0) -- (\rows,0)
        [rotate=60] (-\rows,0) -- (\rows,0);
    \end{scope}
    \begin{scope}[y=(60:1)]
        \coordinate[label= right:$H_{12}$]
        (H12) at (0, \hrows + \cut);
        \coordinate[label= right:$H_{23}$]
        (H23) at (\hrows + \cut,0);
        \coordinate[label= left:$H_{13}$]
        (H13) at (-\hrows - \cut,\hrows +\cut);

        \coordinate[label=center:$123$] (id) at  (30:\reglabel);
        \coordinate[label=center:$213$] (id) at  (90:\reglabel);
        \coordinate[label=center:$231$] (id) at  (150:\reglabel);
        \coordinate[label=center:$321$] (id) at  (210:\reglabel);
        \coordinate[label=center:$312$] (id) at  (270:\reglabel);
        \coordinate[label=center:$132$] (id) at  (330:\reglabel);
    \end{scope}
    \end{tikzpicture}
  \end{center}
  The element $w\in \Sym_{3}$ encodes the region $R_w = \set{x \in V}{ x_{w(1)} \leq x_{w(2)} \leq x_{w(3)}}$.
  Its lattice of flats~$\cL$ is the lattice of set partitions of the set $\{1,2,3\}$.
  Hereby, integers~$i,j$ are in the same block of the corresponding partition if and only if $x_i = x_j$.
  \begin{center}
    \begin{tikzpicture}
        \tikzstyle{vx} = [rectangle, text width=6.5em, align=center]
        \tikzstyle{shortvx} = [rectangle, text width=4em, align=center]
        \tikzstyle{edge} = [shorten <= 0.25pt, shorten >=0.25pt]

        \node[style=shortvx]           (top)   at (0,0) {$\boldsymbol{0}$};
        \node[style=shortvx,right=0em] (13)    at ($(top.west) - (0,4em)$) {$H_{13}$};
        \node[style=shortvx,left=0em]  (12)    at (13.west) {$H_{12}$};
        \node[style=shortvx,right=0em] (23)    at (13.east) {$H_{23}$};

        \node[style=shortvx,right=0em] (cong)   at (23.east) {$\cong$};

        \node[style=vx,right=0em] (b12)   at (cong.east) {$\bigl\{\{1,2\},\{3\}\bigr\}$};
        \node[style=vx,right=0em] (b13)   at (b12.east) {$\bigl\{\{1,3\},\{2\}\bigr\}$};
        \node[style=vx,right=0em] (b23)   at (b13.east) {$\bigl\{\{2,3\},\{1\}\bigr\}$};
        \node[style=vx,right=0em] (btop)  at ($(b13.west) + (0,4em)$) {$\bigl\{1,2,3\bigr\}$};

        \node[style=shortvx,right=0em]	(bot) at ($(13.west) - (0,4em)$) {$V$};
        \node[style=vx,right=0em]	(bbot) at ($(b13.west) - (0,4em)$) {$\bigl\{\{1\},\{2\},\{3\}\bigr\}$};

        \draw[style=edge] (top) -- (12);
        \draw[style=edge] (top) -- (23);
        \draw[style=edge] (top) -- (13);
        \draw[style=edge] (btop) -- (b12);
        \draw[style=edge] (btop) -- (b13);
        \draw[style=edge] (btop) -- (b23);

        \draw[style=edge] (bot) -- (12);
        \draw[style=edge] (bot) -- (23);
        \draw[style=edge] (bot) -- (13);
        \draw[style=edge] (bbot) -- (b12);
        \draw[style=edge] (bbot) -- (b13);
        \draw[style=edge] (bbot) -- (b23);

    \end{tikzpicture}
  \end{center}
  The chains in $\cL \setminus \{V\}$ are given by $\emptyset$, $\{\boldsymbol{0}\}$,  $\{H_{ij}\}$, and $\{H_{ij} < \boldsymbol{0}\}$ for $1\leq i<j \leq 3$.
  We obtain
  \begin{center}
      \renewcommand*{\arraystretch}{1.2}
  \begin{minipage}[t]{0.4\textwidth}
    \begin{tabular}{|l|l|c|}
      \hline
      $\cC$ & $\Poin_\cC (y)$ & $\#\cC$ \\
      \hline
      $\emptyset$ & $1+3y+2y^2$ & 0 \\
      $\{\boldsymbol{0}\}$ & $1+3y+2y^2$ & 1 \\
      $\{H_{ij}\}$ & $(1+y)^2$ & 1 \\
      $\{H_{ij} < \boldsymbol{0}\}$ & $(1+y)^2$ & 2 \\
      \hline
    \end{tabular}
  \end{minipage}
  \begin{minipage}[ct]{0.5\textwidth}
    \begin{align*}
      \operatorname{cfHP}_{\cB_2}(y,t)
      =& \frac{(1+3y+2y^2) + (2+3y+y^2)t}{(1-t)^2}\,.
    \end{align*}
  \end{minipage}
  \end{center}
    \renewcommand*{\arraystretch}{1}
  \noindent\\[0.1em] \noindent The numerator polynomial $\Num_{\cB_2}(y,t) = (1+3y+2y^2) + (2+3y+y^2)t$ evaluated in $y=1$ is
   \[
        \Num_{\cB_2}(1,t) = 6 + 6t = 6 \cdot (1+t) = 6 \cdot \Euler_2(t) \,.
   \]
\end{example}

\subsection{Main result}

Let $(w,\sigma) \in \Sym_{n+1} \times \Perm(n)$ and write~$w = [\ w_1\ |\ w_2\ |\ \dots\ |\ w_{n+1}\ ]$ in one-line notation with bars labeled from left to right by the numbers~$1$ through~$n$ between the entries.
For the given pair $(w,\sigma)$, define a maximal chain in~$\Sigma_n$ by starting with the minimal element given by the all-singleton composition~$w$, and by merging blocks by deleting the bars in the positions in the order given by~$\sigma(1),\sigma(2),\dots,\sigma(n)$.
For~$i \in \{1,\dots,n\}$, one thus has the cover relation in~$\Sigma_n$ given by merging blocks $B | B'$ into $B \cup B'$ at the bar in position~$\sigma(i)$.
The statistic
\[
  \ino : \Sym_{n+1} \times \Perm(n) \rightarrow \NN
\]
is then defined to be the number of such indices~$i$ for which the minimal elements in $B$ and in~$B'$ come \textbf{in o}rder, meaning that $\min(B) < \min(B')$.

\begin{example}
\label{ex:set comp}
  The poset of set compositions of the set $\{1,2,3\}$ is given below.
  Its minimal elements are~$\Sym_3$ and its maximal chains are in bijection with $\Sym_3 \times \Perm(2)$.
  The chain corresponding to $w=123$ and $\sigma=21$ is marked by \textbf{\violet nonviolent violet thick edges} and the chain corresponding to $w=312$ and $\sigma = 12$ is marked by \textbf{\orange golden ratio orange dashed edges}.
  \begin{center}
  \begin{tikzpicture}
      \tikzstyle{vertex} = [rectangle, text width=4em, align=center]
      \tikzstyle{edge} = [shorten <= 0.25pt, shorten >=0.25pt]
      \tikzstyle{whiteedge} = [white, line width=0.5em]
      \tikzstyle{oedge} = [orange, dashed, shorten <= 0.25pt, shorten >=0.25pt, line width=0.2em]
      \tikzstyle{vedge} = [violet, shorten <= 0.25pt, shorten >=0.25pt, line width=0.2em]

      \node[style=vertex]             (top)   at (0,0) {$123$};
      \node[style=vertex,right=0em]   (21)    at ($(top.south) - (0,4em)$) {$3|12$};
      \node[style=vertex,left=0em]   (13)    at (21.west) {$13|2$};
      \node[style=vertex,left=0em]   (23)    at (13.west) {$1|23$};
      \node[style=vertex,left=0em]   (12)    at (23.west) {$12|3$};
      \node[style=vertex,right=0em]   (32)    at (21.east) {$23|1$};
      \node[style=vertex,right=0em]   (31)    at (32.east) {$2|13$};

      \node[style=vertex,right=0em]	(123) at ($(12.west) - (0,5em)$) {$1|2|3$};
      \node[style=vertex,right=0em]	(132) at (123.east)	{$1|3|2$};
      \node[style=vertex,right=0em]	(312) at (132.east)	{$3|1|2$};
      \node[style=vertex,right=0em]	(321) at (312.east)	{$3|2|1$};
      \node[style=vertex,right=0em]	(231) at (321.east)	{$2|3|1$};
      \node[style=vertex,right=0em]	(213) at (231.east)	{$2|1|3$};

      \draw[style=edge] (top) -- (12);
      \draw[style=vedge] (top) -- (23);
      \draw[style=oedge] (top) -- (13);
      \draw[style=edge] (top) -- (21);
      \draw[style=edge] (top) -- (32);
      \draw[style=edge] (top) -- (31);

      \draw[style=edge] ($(12.south) + (0.75em,0)$) -- ($(213.north)-(0.75em,0)$);

      \draw[style=whiteedge] (12) -- (123);
      \draw[style=whiteedge] (23) -- (123);
      \draw[style=whiteedge] (23) -- (132);
      \draw[style=whiteedge] (13) -- (132);
      \draw[style=whiteedge] (13) -- (312);
      \draw[style=whiteedge] (21) -- (312);
      \draw[style=whiteedge] (21) -- (321);
      \draw[style=whiteedge] (32) -- (321);
      \draw[style=whiteedge] (32) -- (231);
      \draw[style=whiteedge] (31) -- (231);
      \draw[style=whiteedge] (31) -- (213);

      \draw[style=edge] (12) -- (123);
      \draw[style=vedge] (23) -- (123);
      \draw[style=edge] (23) -- (132);
      \draw[style=edge] (13) -- (132);
      \draw[style=oedge] (13) -- (312);
      \draw[style=edge] (21) -- (312);
      \draw[style=edge] (21) -- (321);
      \draw[style=edge] (32) -- (321);
      \draw[style=edge] (32) -- (231);
      \draw[style=edge] (31) -- (231);
      \draw[style=edge] (31) -- (213);

  \end{tikzpicture}
  \end{center}
    The merging blocks $3|1$ in the cover relation $3|1|2 \prec 13|2$ are not in order, while the merging blocks in~$13|2 \prec 123$ are.
    Thus,~$\ino({\orange 312},{\orange 12}) = 1$.
    All the blocks in the chain $({\violet 123},{\violet 21})$ come in order, so $\ino({\violet 123},{\violet 21})=2$.
\end{example}

\begin{theorem}
\label{thm:mainresult}
  For the braid arrangement $\cB_n$, we have
  \[
    \Num_{\cB_n}(y,t) = \sum_{\substack{w \in \Sym_{n+1} \\ \sigma \in \Perm(n)}} y^{\ino(w,\sigma)}\cdot t^{\des(\sigma)}\,.
  \]
\end{theorem}

We further refine this theorem in the remainder of this introduction and then prove it in \Cref{sec:proofs}.
Beside its natural appearance, the motivation for studying this companion statistic for the symmetric group is its connection to left-to-right minima for general real reflection groups.

\begin{remark}[Motivation]
\label{rem:generaltype}
  For a real reflection group~$W$, \Cref{qu:companionstatistic} becomes the question of providing a companion statistic for descents on $\Perm(n)$ in terms of the group~$W$,
    \[
      \Num_\cA(y,t) = \sum_{\substack{w \in W \\ \sigma \in \Perm(n)}} y^{\stat(w,\sigma)}\cdot t^{\des(\sigma)}\,.
    \]
  It is known that $\Num_\cA(y,0) = \Poin(\cA,y)$, where the latter is the Poincaré polynomial of the arrangement, see~\cite{MaglioneVoll}.
  For reflection arrangements $\cA_W$, this corresponds to
  \[
    \Poin(\cA_W,y) = \sum_{w \in W} y^{\ell_R(w)} \,,
  \]
  where $\ell_R : W \rightarrow \NN$ denotes the \Dfn{reflection length} on~$W$ given by the minimal number~$k$ such that~$w = r_1 \cdots r_k$ for reflections $r_1,\dots,r_k \in~W$.
  We thus obtain that the \emph{yet-to-be-found} companion statistic on $W \times \Perm(n)$ needs to be equidistributed with the reflection length on the slice~$W \times \{ \id \}$ for the identity permutation $\id = [1,\dots,n] \in \Perm(n)$,
  \[
    \sum_{w \in W} y^{\stat(w,\id)} = \sum_{w \in W} y^{\ell_R(w)} \,.
  \]
  For the symmetric group, it is immediate from the definition that the companion statistic~$\ino(w,\id)$ records exactly the \textbf{non-left-to-right minima} of the permutation~$w$.
  It is therefore a generalization of non-left-to-right minima.
  For general~$W$, a companion statistic~$\stat : W \times \Perm(n)$ would thus necessarily be a version of generalized non-left-to-right minima that is equidistributed on~$W \times \{\id\}$ with reflection length.
\end{remark}

\subsection{Refining the main result}

In order to prove \Cref{thm:mainresult}, we make use of a formula for the numerator polynomial~$\Num_\cA(y,t)$ based on an $R$-labeling (see \Cref{sec:rewritingnumpoly}) on the lattice of flats of $\cB_n$.
The set partition lattice~$\Pi_n$ has an $R$-labeling~$\lambda$ given by labeling a cover relation merging two blocks~$B$ and~$B'$ by $`$, see~\cite[Example~2.9]{MR0570784}.
We extend this $R$-labeling from the set partition lattice to an edge-labeling of the poset of set compositions~$\Sigma_n$ by defining~$\lambda(w,\sigma)=(\lambda_1,\dots,\lambda_n)$ for $w\in \Sym_{n+1}$ and $\sigma \in \Perm(n)$ by
\begin{equation}
\label{labelingsetcompo}
  \lambda_i = \varepsilon(B|B') \cdot \max\big\{\min(B), \min(B')\big\}
\end{equation}
where $B,B'$ are the merging blocks at the $i$-th edge of the chain $(w,\sigma)$,
and $\varepsilon(B|B') \in \{\pm 1\}$ is positive if $\min(B) < \min(B')$ and negative otherwise.
In particular,
\begin{equation}
\label{eq:Ino}
  \ino(w,\sigma) = \Big|\big\{\lambda_i \in \lambda(w,\sigma) \mid \lambda_i > 0\big\}\Big|\,.
\end{equation}
Since these two edge-labelings depend only on the merging blocks, the natural surjection from~$\Sigma_n$ to~$\Pi_n$ (given by forgetting the order of the blocks) maps an edge label to its absolute value.
If not explicitly stated otherwise, we denote by~$\lambda$ the signed version of the edge-labeling.
We obtain the following proposition that we prove in \Cref{sec:rewritingnumpoly}.

\begin{proposition}
\label{prop:numeratorbyRlabelingSym}
    For the braid arrangement $\cB_n$, we have
    \[
    \Num_{\cB_n}(y,t) = \sum_{\substack{w \in \Sym_{n+1} \\ \sigma \in \Perm(n)} }
    y^{\ino(w,\sigma)} \cdot  t^{\asc(\lambda(w,\sigma) )}\,,
    \]
    where $\asc\big(\lambda(w,\sigma) \big) = \big|\{ i \mid \lambda_i < \lambda_{i+1}\}\big|$ is the number of ascents of $\lambda(w,\sigma) = (\lambda_1,\dots,\lambda_n)$.
\end{proposition}

This description looks similar to \Cref{thm:mainresult}.
But this is not quite the case since here, both statistics depend on~$w$ and~$\sigma$, so this is \textbf{not} a companion statistic for descents (or ascents) on~$\Perm(n)$.
We in particular emphasize that we do \emph{not have a bijection}~$\varphi$ on $\Sym_{n+1} \times \Perm(n)$ sending $\asc\big(\lambda(w,\sigma) \big)$ to $\des\big(\varphi(\sigma)\big)$ while leaving $\ino(w,\sigma) = \ino\big(w,\varphi(\sigma)\big)$ invariant.

\begin{example}
  For~$\Sigma_2$, we obtain the edge-labeling
  \begin{center}
    \begin{tikzpicture}
        \tikzstyle{vertex} = [rectangle, text width=4em, align=center]
        \tikzstyle{edge} = [shorten <= 0.25pt, shorten >=0.25pt]
        \tikzstyle{whiteedge} = [white, line width=0.5em]
        \tikzstyle{oedge} = [orange, dashed, shorten <= 0.25pt, shorten >=0.25pt, line width=0.2em]
        \tikzstyle{vedge} = [violet, shorten <= 0.25pt, shorten >=0.25pt, line width=0.2em]

        \node[style=vertex]             (top)   at (0,0) {$123$};
        \node[style=vertex,right=0em]   (21)    at ($(top.south) - (0,4em)$) {$3|12$};
        \node[style=vertex,left=0em]   (13)    at (21.west) {$13|2$};
        \node[style=vertex,left=0em]   (23)    at (13.west) {$1|23$};
        \node[style=vertex,left=0em]   (12)    at (23.west) {$12|3$};
        \node[style=vertex,right=0em]   (32)    at (21.east) {$23|1$};
        \node[style=vertex,right=0em]   (31)    at (32.east) {$2|13$};

        \node[style=vertex,right=0em]	(123) at ($(12.west) - (0,5em)$) {$1|2|3$};
        \node[style=vertex,right=0em]	(132) at (123.east)	{$1|3|2$};
        \node[style=vertex,right=0em]	(312) at (132.east)	{$3|1|2$};
        \node[style=vertex,right=0em]	(321) at (312.east)	{$3|2|1$};
        \node[style=vertex,right=0em]	(231) at (321.east)	{$2|3|1$};
        \node[style=vertex,right=0em]	(213) at (231.east)	{$2|1|3$};

        \draw[style=edge] (top) -- (12) node [midway,above left]
        {\scriptsize $3$};
        \draw[style=vedge] (top) -- (23) node [midway, left]
        {\scriptsize $2 \ $};
        \draw[style=oedge] (top) -- (13) node [midway, left]
        {\scriptsize $2$};
        \draw[style=edge] (top) -- (21) node [midway, right]
        {\scriptsize $\overline{3}$};
        \draw[style=edge] (top) -- (32) node [midway, right]
        {\scriptsize $\ \overline{2}$};
        \draw[style=edge] (top) -- (31) node [midway,above right]
        {\scriptsize $\overline{2}$};

        \draw[style=edge] ($(12.south) + (0.75em,0)$) -- ($(213.north)-(0.75em,0)$) node [above right, pos=0.95] {\scriptsize $\overline{2}$};

        \draw[style=whiteedge] (12) -- (123);
        \draw[style=whiteedge] (23) -- (123);
        \draw[style=whiteedge] (23) -- (132);
        \draw[style=whiteedge] (13) -- (132);
        \draw[style=whiteedge] (13) -- (312);
        \draw[style=whiteedge] (21) -- (312);
        \draw[style=whiteedge] (21) -- (321);
        \draw[style=whiteedge] (32) -- (321);
        \draw[style=whiteedge] (32) -- (231);
        \draw[style=whiteedge] (31) -- (231);
        \draw[style=whiteedge] (31) -- (213);

        \draw[style=edge] (12) -- (123) node [midway,above left]
        {\scriptsize $2$};
        \draw[style=vedge] (23) -- (123) node [midway, left]
        {\scriptsize $3$};
        \draw[style=edge] (23) -- (132) node [midway,left]
        {\scriptsize $\overline{3}$};
        \draw[style=edge] (13) -- (132) node [midway, left]
        {\scriptsize $3$};
        \draw[style=oedge] (13) -- (312) node [midway, below left]
        {\scriptsize $\overline{3}$};
        \draw[style=edge] (21) -- (312) node [near end, right]
        {\scriptsize $2$};
        \draw[style=edge] (21) -- (321) node [midway,above right]
        {\scriptsize $\overline{2}$};
        \draw[style=edge] (32) -- (321) node [midway,above left]
        {\scriptsize $\overline{3}$};
        \draw[style=edge] (32) -- (231) node [midway,above left]
        {\scriptsize $3$};
        \draw[style=edge] (31) -- (231) node [midway,above left]
        {\scriptsize $\overline{3}$};
        \draw[style=edge] (31) -- (213) node [midway, right]
        {\scriptsize $3$};
    \end{tikzpicture}
  \end{center}
  (Here and below, we often write $\overline{x} = -x$ for $x \in \ZZ$.)
  The sequences $\lambda(w,\sigma)$ for the chains $({\violet 123},{\violet 21})$ and $({\orange 312},{\orange 12})$ are in particular
  \[
    \lambda({\violet 123},{\violet 21})= (3,2)
    \quad \text{and} \quad
    \lambda({\orange 312},{\orange 12})= (\hb{3},2)\,.
  \]
  We obtain $\asc\big(\lambda ( {\violet 123},{\violet 21}) \big) = 0$ and $\asc\big( \lambda({\orange 231},{\orange 12}) \big) = 1$.
\end{example}

The following theorem refines \Cref{thm:mainresult}.

\begin{theorem}
    \label{thm:mainresult local}
    Let~$w \in \Sym_{n+1}$, then
    \begin{align*}
        \sum_{\sigma \in \Perm(n) }
        y^{\ino(w,\sigma)} \cdot t^{\asc(\lambda(w,\sigma) )}
        &= \sum_{\sigma \in \Perm(n) }
        y^{\ino(w,\sigma)} \cdot t^{\des(\sigma)},
    \end{align*}
\end{theorem}

We prove a multivariate version of this result as stated in \Cref{thm:mainresult localocal} in \Cref{sec:proofs}.

\begin{remark}
    The right-hand side of \Cref{thm:mainresult} depends on multiple choices.
    The ``max-of-min'' labeling in~\eqref{labelingsetcompo} could have also been ``min-of-max'', the ``in-order'' statistic~$\ino(w,\sigma)$ could have been an ``out-of-order'' statistic, and descents on~$\sigma$ could have been ascents.
    While multiple combinations yield the same generating function in  \Cref{thm:mainresult}, the given version is the \textbf{unique choice} for which the refinement in \Cref{thm:mainresult local} holds.
\end{remark}

\begin{remark}
  As written in \cite[Example~2.9]{MR0570784}, the ``max-of-min'' labeling in~\eqref{labelingsetcompo} is the edge-labeling given in~\cite[Proposition~2.4]{MR0354473} for supersolvable lattices, using the~\emph{$\cM$-chain}
    \[
      1\ |\ 2\ |\ 3\ |\ \cdots\ |\ n+1 \prec 12\ |\ 3\ |\  \cdots \ |\  n+1 \prec \dots \prec 1\cdots n\ |\ n+1 \prec 1\cdots n+1
    \]
    in $\Pi_n$.
  This $\cM$-chain induces the admissible map $H_{ij} \mapsto j$, which in turn induces the ``max-of-min'' $R$-labeling.
  This admissible map, and thus the ``max-of-min'' $R$-labeling, may as well be obtained by considering the induced edge-labeling of an admissible map given in~\cite[Proposition~2.2]{MR0354473} for geometric lattices, using the total order
    \[
      H_{12} < H_{13} < H_{23} < H_{14} < H_{24} < H_{34} < \dots < H_{1(n+1)} < \dots < H_{n(n+1)}
    \]
    of the atoms in~$\Pi_n$.
    It is easy to check that no two labels~$H_{ij}$ and~$H_{i'j}$ can occur along a maximal chain in~$\Pi_n$, implying again that~$H_{ij} \mapsto j$ is an admissible map.
\end{remark}

\section{Proofs}
\label{sec:proofs}

In this section, we
\begin{enumerate}
  \item prove \Cref{prop:numeratorbyRlabelingSym} (\Cref{sec:rewritingnumpoly}),
  \item provide the most refined version of the main result in \Cref{thm:mainresult localocal} (\Cref{subsec: main result refined}),
  \item collect several properties of the $R$-labeling and reestablish a symmetry of $\Num_{\cB_n}$ (\Cref{sec: edge-labeling}),
  \item define a poset associated to $w \in \Sym_{n+1}$ and certain subsets $Y \subseteq \{1,\dots,n\}$ together with an injective vertex-labeling (\Cref{sec: poset}),
  \item use reverse $(P,\omega)$-partitions to prove the main result (\Cref{subsec: proof mainresult}).
\end{enumerate}

\subsection{Rewriting the numerator polynomial}
\label{sec:rewritingnumpoly}

Let~$P$ be a finite and ranked poset of rank~$n$ with unique minimal element~$\widehat{0}$ and unique maximal element~$\widehat{1}$.
Denote its set of cover relations by~$E(P) = \set{(x,y)}{ x \prec y }$.
An edge-labeling $\lambda: E(P) \to \NN_+$ is an \Dfn{$\boldsymbol{R}$-labeling} if every interval in~$P$ contains a unique $\lambda$-weakly increasing maximal chain.
Throughout this section, we assume the edge-labeling to be injective along maximal chains.
For a maximal chain~$\cM = \{ \cM_0 \prec \dots \prec \cM_n\}$, define its labeling as $\lambda(\cM)=\big(\lambda(\cM_0,\cM_1), \dots, \lambda(\cM_{n-1},\cM_n)\big)$.
For a subset $Y \subseteq \{1,\dots,n\}$, we slightly modify $\lambda(\cM)$ to obtain $\lambda(\cM,Y) = \big(\!\pm \lambda(\cM_0,\cM_1), \dots, \pm \lambda(\cM_{n-1},\cM_n)\big)$ where the sign is positive if $i \in Y$ and negative if $i \notin Y$.

\medskip

The lattice of flats~$\cL$ of a hyperplane arrangement~$\cA$ always admits an $R$-labeling~$\lambda$~\cite[Example 3.14.5]{MR2868112}.
The numerator polynomial $\Num_\cA$ was shown in~\cite{poincareextended} to have the description
\begin{equation}
\label{eq: Num as sum over ME}
  \Num_\cA( y, t) = \sum_{\cM,Y} y^{\#Y} \cdot t^{\asc( \lambda(\cM,Y) ) }
\end{equation}
where the sum ranges over all maximal chain $\cM$ in~$\cL$ and all subsets~$Y \subseteq \{1,\dots, n\}$ for~$n = \rk(\cL)$, and where~$\asc\big(\lambda(\cM,Y)\big)$ denotes the number of ascents of $\lambda(\cM,Y)$.
We refer to~\cite{poincareextended} for further details.

\begin{example}\label{ex:NumWithME}
    Label the edges in the lattice of flats of the braid arrangement $\cB_2$ by the $R$-labeling~$\lambda$ defined above.
    The table on the right gives the relation between the signed labels along the three maximal chains $\cM_{ij} = \{V \prec H_{ij} \prec \boldsymbol{0}\}$, depending on the subsets~$Y \subseteq \{1,2\}$.
    \begin{center}
        \begin{minipage}[ct]{0.45\textwidth}
            \begin{tikzpicture}
                \tikzstyle{vertex} = [rectangle, text width=4em, align=center]
                \tikzstyle{edge} = [shorten <= 0.25pt, shorten >=0.25pt]

                \node[style=vertex]           (top)   at (0,0) {$\boldsymbol{0}$};
                \node[style=vertex,right=0em] (H13)   at ($(top.west) - (0,4em)$) {$H_{13}$};
                \node[style=vertex,left=0em]  (H12)   at (H13.west) {$H_{12}$};
                \node[style=vertex,right=0em] (H23)   at (H13.east) {$H_{23}$};
                \node[style=vertex,right=0em]	(bot)   at ($(H13.west) - (0,4em)$) {$V$};

                \draw[style=edge] (top) -- (H12) node [midway,above left]
                {\scriptsize $3$};
                \draw[style=edge] (top) -- (H23) node [midway,above right]
                {\scriptsize $2$};
                \draw[style=edge] (top) -- (H13) node [midway,right]
                {\scriptsize $2$};
                \draw[style=edge] (bot) -- (H12) node [midway,below left]
                {\scriptsize $2$};
                \draw[style=edge] (bot) -- (H23) node [midway,below right]
                {\scriptsize $3$};
                \draw[style=edge] (bot) -- (H13) node [midway,right]
                {\scriptsize $3$};
            \end{tikzpicture}
        \end{minipage}
        \begin{minipage}[ct]{0.45\textwidth}
            \renewcommand*{\arraystretch}{1.2}
            \begin{tabular}{|c|c|c|c|c|}
                \hline
                $Y$ & $\emptyset$ & $\{1\}$ & $\{2\}$ & $\{1,2\}$ \\
                \hline
                $\cM_{12}$ & $-2>-3$ & $2>-3$ & $-2< 3$ & $2<3$ \\[0.25em]
                \hline
                $\cM_{13}$ & \multirow{2}{*}{$-3<-2$} & \multirow{2}{*}{$3 > -2$} & \multirow{2}{*}{$-3<2$} & \multirow{2}{*}{$3>2$} \\[0.25em]
                $\cM_{23}$ & & & & \\
                \hline
            \end{tabular}
        \end{minipage}
    \end{center}
    Using~\eqref{eq: Num as sum over ME}, the numerator polynomial for~$\cB_2$ is
    \begin{align*}
        \Num_{\cB_2}(y,t) = y^0 \cdot (t+2) + y^1 \cdot (3t +3) + y^2 \cdot (2t+1)
        = (1+3y+2y^2) + (2 + 3y + y^2)t
    \end{align*}
    in agreement with \Cref{ex:coarseflagHP}.
\end{example}

For a simplicial arrangement~$\cA$, the set $\regions \times \Perm(n)$ is in natural bijection with maximal chains in the \Dfn{face poset}~$\Sigma$ of~$\cA$.
The elements of this poset are the faces of~$\cA$, i.e., the cones in the simplicial fan defined by~$\cA$, and the order is reverse inclusion of cones.
The minimal elements of~$\Sigma$ are in particular the closures of the regions~$\cR$.
Sending such a maximal chain in the face poset to its corresponding maximal chain in the lattice of flats gives a natural bijection
\begin{equation}
\label{eq:faceposetbij}
  \regions \times \Perm(n)\ \tilde\longleftrightarrow\ \big\{ \text{maximal chains in } \cL\big\} \times \big\{ \text{subsets of } \{1,\dots,n\}\big\}\,.
\end{equation}
We have already seen this correspondence in the case of the braid arrangement~$\cB_n$.
The bijection between $\Sym_{n+1} \times \Perm(n)$ and maximal chains in~$\Sigma_n$ was described in the paragraph before and inside \Cref{ex:set comp}.
Moreover, sending a set composition to its corresponding set partition yields the subset $Y \subseteq \{1,\dots,n\}$ given by $Y = \big\{i \in \{1,\dots,n\} \mid \lambda_i \text{ positive} \big\}$ for~$\lambda_i$ as given in~\eqref{labelingsetcompo}.

\begin{proof}[Proof of \Cref{prop:numeratorbyRlabelingSym}]
  Consider the bijection in~\eqref{eq:faceposetbij} in the case of the braid arrangement as described in the previous paragraph.
  Let now $(w,\sigma) \in \Sym_{n+1} \times \Perm(n)$, and let $(\cM,Y)$ be its image under the bijection.
  Then, by construction, we have $\ino(w,\sigma) = \#Y$ and at the same time~$\lambda(w,\sigma) = \lambda(\cM,Y)$.
\end{proof}

\subsection{Further refinement of the main result}
\label{subsec: main result refined}

For $(w,\sigma)\in \Sym_{n+1} \times \Perm(n)$,
let
\[
  \Ino(w,\sigma)=\set{\lambda_i \in \lambda(w,\sigma)}{\lambda_i > 0}
\]
denote the set of positive integers in $\lambda(w,\sigma)$.
We have seen this set already in~\eqref{eq:Ino} where we observe that $\ino(w,\sigma) = |\Ino(w,\sigma)|$.
Moreover, let~$\Des(\sigma)=\bigset{i\in \{1,\dots, n-1\} }{\sigma(i) > \sigma(i+1)}$ denote the \Dfn{descent set} of $\sigma$ and similarly, let~$\Asc(\lambda\big(w,\sigma)\big)$ denote the \Dfn{ascent set} of $\lambda(w,\sigma)$.
We finally consider refined indeterminates~$y_1,\dots,y_{n},t_1,\dots,t_{n-1}$ and write
\[
\boldsymbol{y}^Y = \prod_{i\in Y} y_i \,, \quad \boldsymbol{t}^T = \prod_{i\in T} t_i
\]
for subsets~$Y \subseteq \{1,\dots, n\}$ and~$T\subseteq \{1,\dots, n-1\}$.

\begin{theorem}
\label{thm:mainresult localocal}
  Let $w\in \Sym_{n+1}$, then
  \[
  \sum_{\sigma \in \Perm(n)}
  \boldsymbol{y}^{\Ino(w,\sigma)} \cdot
  \boldsymbol{t}^{\Asc(\lambda(w,\sigma) )}
  =
  \sum_{\sigma \in \Perm(n)}
  \boldsymbol{y}^{\Ino(w,\sigma)} \cdot
  \boldsymbol{t}^{\Des(\sigma)} \,.
  \]
\end{theorem}

\begin{example}
  In our running example $\cB_2$, we have
  \begin{center}
  \resizebox{\textwidth}{!}{%
  \begin{tabular}{c||cc|cc|cc|cc}
    $w$ & $\lambda(w,12)$ & $\lambda(w,21)$
    & $\boldsymbol{y}^{\Ino(w,12)}$
    & $\boldsymbol{y}^{\Ino(w,21)}$ & $\boldsymbol{t}^{\Asc(\lambda(w,12))}$
    &$\boldsymbol{t}^{\Asc(\lambda(w,21))}$ & $\boldsymbol{t}^{\Des(12)}$
    & $\boldsymbol{t}^{\Des(21)}$  \\
    \hline
    123 & 23 & 32 & $y_2 y_3$ & $y_2 y_3$
        & $t_1$ & $1$
        & $1$ & $t_1$ \\
    132 & 32 & $\hb{3}2$ & $y_2 y_3$ & $y_2$
        & $1$ & $t_1$
        & $1$ & $t_1$  \\
    213 & $\hb{2}3$ & $3\hb{2}$ & $y_3$ & $y_3$
        & $t_1$ & $1$
        & $1$ & $t_1$  \\
    231 & $3\hb{2}$ & $\hb{3}\hb{2}$ & $y_3$ & $1$
        & $1$ &$t_1$
        & $1$ & $t_1$ \\
    312 & $\hb{3}2$ & $2\hb{3}$ & $y_2$ & $y_2$
        & $t_1$ & $1$
        & $1$ & $t_1$  \\
    321 & $\hb{3}\hb{2}$ & $\hb{2}\hb{3}$ & $1$ & $1$
        & $t_1$ & $1$
        & $1$ & $t_1$
  \end{tabular}
  }
  \end{center}
  To give one concrete entry, the equation in \Cref{thm:mainresult localocal} for $w=213$ is
  \[
      y_3 t_1 + y_3  = y_3 + y_3 t_1 \,.
  \]
\end{example}

\subsection{Behavior of the edge-labeling $\boldsymbol{\lambda}$}
\label{sec: edge-labeling}

Throughout this section, fix $w = w_1\dots w_{n+1}\in \Sym_{n+1}$ and $\sigma\in \Perm(n)$.
Moreover, recall the definition of~$\lambda(w,\sigma)$ given in~\eqref{labelingsetcompo}.
We start with the following two lemmas, for which we set $\overleftarrow{w}=w_{n+1} \cdots w_1 \in \Sym_{n+1}$ sending~$i$ to $w_{n+2-i}$ and ${\sigma_{\uparrow} \in \Perm(n)}$ sending~$i$ to~$n+1-\sigma(i)$.

\begin{lemma}\label{lma:for symmetry}
    Let~$\lambda(w, \sigma) = (\lambda_1 , \dots , \lambda_{n})$.
    Then
    \[
        \lambda(\overleftarrow{w}, \sigma_{\uparrow})
        = (-{\lambda_1} , \dots , -{\lambda_n}) \,.
    \]
\end{lemma}
\begin{proof}
    The~$i$-th cover relation in~$(w,\sigma)$ merges the exact same blocks as the $i$-th cover relation in~$(\overleftarrow{w}, \sigma_{\uparrow})$. 
    The merging blocks only come in opposite order, swapping the sign of~$\lambda_i$.
\end{proof}

\begin{lemma}
\label{lma: edge label behaviour}
  For~$\lambda(w, \sigma) = (\lambda_1 , \dots , \lambda_{n})$, we have
  \[
      \big\{ |\lambda_1| ,\dots, |\lambda_n| \big\} = \big\{ 2, \dots, n+1 \big\} \,.
  \]
\end{lemma}
\begin{proof}
    This follows immediately from the observation that every number in $\{2,\dots,n+1\}$ is exactly once the maximum of the minima of two merging blocks.
    (One may observe in addition that the number~$1$ is never such a maximum.)
\end{proof}

Combining these two lemmas yields the following symmetry on~¸$\Num_{\cB_n}$.
More general forms of this symmetry have already been observed in ~\cite[Corollary~3.3]{MaglioneVoll} and~\cite[Corollary~2.16]{poincareextended}.

\begin{corollary}
\label{lma:symmetry}
    It holds that
    \begin{align*}
        \Ino(\overleftarrow{w},\sigma_{\uparrow}) 
        &= 
        \{2,\dots , n+1\} \setminus \Ino(w,\sigma) \\
        \Asc\big(\lambda(\overleftarrow{w},\sigma_{\uparrow})\big)
        &= 
        \{2,\dots , n+1\} \setminus \Asc\big(\lambda(w,\sigma)\big)
        \,.
    \end{align*}
\end{corollary}

For an entry $w_i\in \{1,\dots, n+1\}$ at position $i$ in the one-line notation of $w$, define two positions~$\ell < i < r$ by
\begin{equation}
    \label{eq: left right bar pos}
    \ell = \max\bigset{ j < i }{ w_j < w_i } \,, \quad
    r    = \min\bigset{ j > i }{ w_j < w_i } \,.
\end{equation}
In words,~$\ell$ is the rightmost position left of the position $i$ such that $w_\ell$ is smaller than $w_i$.
Similarly,~$r$ is the leftmost position right of the position $i$ such that $w_r$ is smaller than~$w_i$.
\[
w = \cdots \ w_\ell
\underbrace{\phantom{ sp }\cdots\phantom{ sp }}_{>w_i}
\ w_i \
\underbrace{\phantom{ sp }\cdots\phantom{ sp }}_{>w_i}
w_r \ \cdots
\]
Observe that $\ell= - \infty$ if and only if $w_i$ is a \Dfn{left-to-right minimum} of $w$,
and $r = \infty$ if and only if $w_i$ is a \Dfn{right-to-left minimum}.

We also record the cover relation ranks in~$\Sigma_n$ in which the blocks containing $w_\ell$ and $w_i$, respectively the blocks containing $w_r$ and $w_i$ are merged.
In symbols,
\begin{align*}
    \ell_\sigma &= \max \bigset{ k\in \{1,\dots, n\} }{ \sigma_k \in \{\ell,\dots, i-1 \} } \,, \\
    r_\sigma &= \max \bigset{ k\in \{1,\dots, n\} }{ \sigma_k \in \{ i,\dots, r-1 \} } \,.
\end{align*}
We call $\ell_\sigma$ the \Dfn{last left bar position} of $w_i$ and $r_\sigma$ the \Dfn{last right bar position} of $w_i$.
If either of the two sets in~\eqref{eq: left right bar pos} is empty, the corresponding last (left or right) bar position is set to be infinity.
The reason for this choice of convention is that $\ell_\sigma$ is the rank in which $w_i$ is merged with something smaller in $w$ to the left, and thus set to be infinity if $w_i$ is a left-to-right minimum.
The same reason holds for $r_\sigma$ if $w_i$ is a right-to-left minimum.

\begin{example}
    Let $w=215463 \in \Sym_{6}$ and $\sigma = 14253\in \Perm_{5}$.
    The corresponding maximal chain in~$\Sigma_n$ and its edge-labeling are
    \begin{center}
        \begin{tabular}{ccccccccccc}
            $2|1|5|4|6|3$ & $\prec$ & $12|5|4|6|3$& $\prec$ &$12|5|46|3$& $\prec$& $125|46|3$& $\prec$& $125|346$& $\prec$& $123456 \,,$ \\

            $\lambda(w,\sigma) = ($\phantom{$2$} & $\hb{2},$ & \phantom{$24|1|5|6|3$}&
            $6,$ &\phantom{$24|1|56|3$}& $5,$& \phantom{$124|56|3$}& $\hb{4},$& \phantom{$12456|3$}& $3$& $)\,.$\phantom{$123456 \,.$}
        \end{tabular}
    \end{center}
    The left-to-right minima of $w$ are $\{2,1\}$ and the right-to-left minima are $\{3,1\}$.

    \medskip

    In the picture below, we label the bars in~$w$ by $\sigma^{-1}$, meaning that the bars are removed in this order.
    For each entry~$w_i$, the last left bar position~$\ell_\sigma$ and the last right bar position~$r_\sigma$ are written above in {\red red}, with~$\ell_\sigma$ on the left and~$r_\sigma$ on the right.
    \begin{center}
      \includegraphics{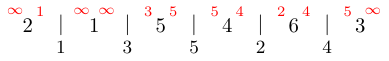}
    \end{center}
\end{example}

The following lemma follows immediately from the above observations.

\begin{lemma}
\label{lma:lambdafromlr}
    Let $\lambda(w,\sigma) = (\lambda_1 ,\dots \lambda_n)$, let $j\in \{1,\dots, n\}$ and $i \in \{1,\dots, n+1\}$.
    Moreover, set~$\ell_\sigma$ and~$r_\sigma$ to be the last (left and, respectively, right) bar positions of~$w$ for the entry~$w_i$.
    Then
    \[
        \lambda_j =
        \begin{cases}
            w_i        & \text{if } j=\ell_\sigma < r_\sigma \,, \\[5pt]
           \hb{w_i}   & \text{if } j=r_\sigma < \ell_\sigma \,.
        \end{cases}
    \]
\end{lemma}

\begin{example}
    We continue the previous example.
    For each entry~$w_i$, we write above in {\red red} the last left and right bar positions. We moreover circle their respective minima.
    \begin{figure}[h]
        \includegraphics{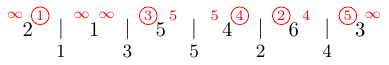}
    \end{figure}

    \noindent Using \Cref{lma:lambdafromlr}, we recompute
    \[
      \lambda(215463,14253) = (\lambda_1, \dots , \lambda_5) = (\hb 2, 6, 5, \hb 4, 3)\,.
    \]
    In particular, every position in $\{1,\dots,5\}$ is circled exactly once in agreement with \Cref{lma: edge label behaviour}, and the signs of~$\lambda$ are determined whether the last (left or right) bar position is smaller.
\end{example}

\subsection{The poset $\boldsymbol{P_{w,Y}}$}
\label{sec: poset}

Throughout this section, fix $w \in \Sym_{n+1}$ and $Y \subseteq \{1,\dots,n\}$ such that
\begin{itemize}
  \item $Y$ does not contain any left-to-right minima of~$w$, while
  \item $Y$ contains all right-to-left minima of~$w$, except the number~$1$.
\end{itemize}
Based on the properties of the edge-labeling $\lambda(w,\sigma)$ provided in the previous section, we define a poset $P = P_{w,Y}$ on the set $\{1,\dots,n\}$ and an injective vertex-labeling~$\Lambda$ of~$P$.
This poset is built so that
\begin{equation*}
    \label{eq: Lin and Ino}
  \Lin(P) = \{\sigma \in \Perm(n) \mid \Ino(w,\sigma) = Y \}\,,
\end{equation*}
where $\Lin(P)$ denotes the set of all linear extensions.
In particular, this implies that $\Perm(n)$ is equal to the disjoint union of the linear extensions of the posets $P_{w,Y}$ for all possible subsets~$Y \subseteq~\{1,\dots,n\}$ with the two given properties.

\medskip

We construct the poset~$P$ together with its injective vertex-labeling~$\Lambda$ by iterating through the numbers $n+1,\dots,2$ in this order.
At each iteration step $w_i \in\{ n+1,\dots,2\}$, we set the label~$\Lambda$ of one vertex.
If~$w_i$ is neither a left-to-right nor right-to-left minimum of~$w$, we add an additional cover relation to~$P$, based on the already added cover relations.

\begin{enumerate}
    \item If $w_i$ is not a left-to-right minimum, we have $\ell \geq 1$ and the set $\{\ell,\dots,i-1\}$ forms a connected subposet of~$P$ with unique maximum. Denote this maximum by~$a$.
    \item If $w_i$ is not a right-to-left minimum, we have $r \leq n$ and the set $\{i,\dots,r-1\}$ forms a connected subposet of~$P$ with unique maximum.
    Denote this maximum by~$b$.\\
    \item If $w_i$ is a left-to-right minimum, it is not a right-to-left minimum. Set $\Lambda(b) = \hb{w_i}$.
    \item If $w_i$ is a right-to-left minimum, it is not a left-to-right minimum. Set $\Lambda(a) = w_i$.
    \item If $w_i$ is neither a left-to-right nor a right-to-left minimum, add the cover relation together with the vertex label of the minimum given by
    \begin{align*}
      a \prec_P b \text{ and } \Lambda(a) =     w_i  \qquad &\text{ if } w_i    \in Y \,\\[5pt]
      b \prec_P a \text{ and } \Lambda(b) = \hb{w_i} \qquad &\text{ if } w_i \notin Y \,.
    \end{align*}
\end{enumerate}

Observe that this iterative process is well-defined since the constructed Hasse diagram is a forest at every step.

\begin{example}
  We continue the previous example with $w= 215463$, and let $Y = \{3,5,6\}$ satisfying the two required properties
  \[
    \{3\} \subseteq Y \subseteq \{3,4,5,6\}\,.
  \]
  The iterative construction of the poset~$P$ is indicated in the following figure row-by-row and in each row left-to-right.
  At each step, there is set one additional vertex-labeling~$\Lambda$, and in the first three iteration steps, there is also a cover relation added.
  This corresponds to the observation that the entries $6,5,4$ in~$w$ are neither left-to-right nor right-to-left minima, while $3,2$ are.
  \begin{center}
    \includegraphics[width=0.7\textwidth]{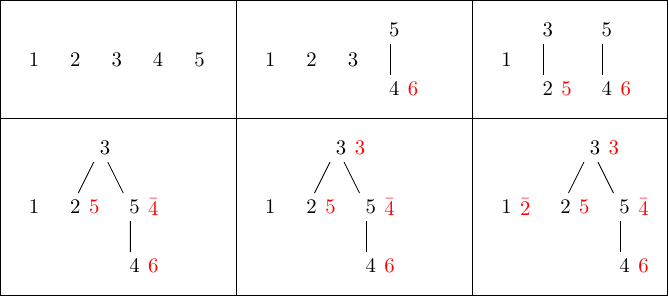}
  \end{center}
\end{example}

\begin{proposition} \label{lma: Lin in P}
    Let $\sigma \in \Perm(n)$ and let $Y = \Ino(w,\sigma)$.
    Then, $\sigma \in \Lin(P_{w,Y}).$
    Moreover,
    \[
      \lambda(w,\sigma) = \Lambda(\sigma) =
      \Big(\Lambda\big(\sigma(1)\big),\dots,\Lambda\big(\sigma(n)\big)\Big)\,,
    \]
    where~$\Lambda$ is the injective vertex-labeling of~$P_{w,Y}$.
\end{proposition}

\begin{proof}
  If an iteration step $w_i \in \{n+1,\dots,2\}$ introduces a new cover relation, the two indices~$a$ and~$b$ are given by $a = \sigma(\ell_\sigma)$ and $b = \sigma(r_\sigma)$.
  Then, $a \prec_P b$ if $\ell_\sigma < r_\sigma$ and $b \prec_P a$ if $r_\sigma < \ell_\sigma$.
  In particular, $a \prec_P b$ if and only if $\sigma^{-1}(a) < \sigma^{-1}(b)$.
  We conclude that $\sigma \in \Lin(P)$.
  The second property follows with \Cref{lma:lambdafromlr}.
\end{proof}

\begin{example}
    In our running example, $\sigma = 14253\in \Lin(P)$ is a linear extension of $P$ and we obtain
    \[
      \Lambda(\sigma) = (\hb{2}, 6, 5, \hb{4}, 3)=\lambda(w,\sigma) \,.
    \]
\end{example}

\subsection{Quasisymmetric functions and a proof of the main result}
\label{subsec: proof mainresult}

In order to prove \Cref{thm:mainresult localocal}, we use the following implication of the theory of reverse $(P,\omega)$-partitions from \cite[§7.19]{MR4621625}.

\begin{proposition}
\label{prop:Pomegapartitions}
  Let $P$ be a poset on $\{1,\dots,n\}$ and let $\omega,\omega': P\to \ZZ$ be two injective vertex-labelings having the same relative orders on the cover relations.
  This is, for each cover relation~$a \prec_P~b$ we have
  \[
    \omega(a) < \omega(b) \Leftrightarrow \omega'(a) < \omega'(b)\,.
  \]
  Then, $\omega$ and~$\omega'$ have the same descent set generating function on the set of linear extensions,
  \[
    \sum_{\sigma \in \Lin(P)} \boldsymbol{t}^{\Des(\omega(\sigma))} = \sum_{\sigma \in \Lin(P)} \boldsymbol{t}^{\Des(\omega'(\sigma))}\,,
  \]
  where $\Des\big( \omega(\sigma)\big)$ denotes the descent set of $\big(\omega(\sigma(1)),\dots,\omega(\sigma(n))\big)$ and analogously for $\Des(\omega'(\sigma))$.
\end{proposition}

Before proving this proposition, we deduce \Cref{thm:mainresult localocal} as follows.

\begin{lemma}
\label{lma:reversecovers}
    For a cover relation $a \prec_P b$ in $P = P_{w,Y}$, we have
    \[
      a < b \Longleftrightarrow \Lambda(a) > \Lambda(b)\,,
    \]
    where both comparisons are as integers.
\end{lemma}
\begin{proof}
    Within the construction of $P$, a cover relation $a\prec_P b$ is introduced for $a < b$ if~$a$ and~$b$ are the last left and right bars for some $w_i\in \{n+1, \dots, 2\}$ and $\Lambda(a)=w_i$.
    If a cover relation~$a\prec_P b$ is introduced for $b < a$ instead, the label of $a$ is set to $\Lambda(a)=\hb{w_i}$.
    In both cases,~$\Lambda(b)$ will be determined at a later iteration step.
    Thus, we have $\hb{w_{i}} < \Lambda(b) < w_{i}$,
    so that $\Lambda(a) > \Lambda(b)$ if $a < b$, and $\Lambda(a) < \Lambda(b)$ if $a > b$ as desired.
\end{proof}

\begin{proof}[Proof of \Cref{thm:mainresult localocal}]
    Using \Cref{lma: Lin in P} and \Cref{lma:reversecovers}, it follows with \Cref{prop:Pomegapartitions} that
    \begin{align}
    \label{eq: key eq for proof}
      \sum_{\sigma \in \Lin(P_{w,Y}) } \boldsymbol{t}^{\Des(\sigma)}
      =
      \sum_{\sigma \in \Lin(P_{w,Y}) }
      \boldsymbol{t}^{\Asc(\lambda(w,\sigma))} \,.
    \end{align}
    Since $\Perm(n)$ is the disjoint union over all linear extensions of the posets $P_{w,Y}$, the observation that $\Ino(w,\sigma) = Y$ for $\sigma \in \Lin(P_{w,Y})$ applied to~\eqref{eq: key eq for proof} implies
    \begin{align*}
        \sum_{\sigma \in \Perm(n)}
        \boldsymbol{y}^{\Ino(w,\sigma)} \cdot
        \boldsymbol{t}^{\Asc(\lambda(w,\sigma) )}
        &=
        \sum_{Y}
        \boldsymbol{y}^{Y}
        \sum_{\sigma \in \Lin(P_{w,Y}) }
        \boldsymbol{t}^{\Asc(\lambda(w,\sigma))}              \\
        &=
        \sum_{Y}
        \boldsymbol{y}^{Y}
        \sum_{\sigma \in \Lin(P_{w,Y}) }
        \boldsymbol{t}^{\Des(\sigma)}
        =
        \sum_{\sigma \in \Perm(n)}
        \boldsymbol{y}^{\Ino(w,\sigma)} \cdot
        \boldsymbol{t}^{\Des(\sigma)} \,,
    \end{align*}
  where the two outer sums on the right-hand side range over all subsets $Y \subseteq \{1,\dots,n\}$ that do not contain left-to-right minima of~$w$ while containing all right-to-left minima of~$w$ except the number~$1$.
\end{proof}

We are left to provide a proof of \Cref{prop:Pomegapartitions}.
Let~$P$ be a finite poset on~$\{1,\dots, n\}$ and let~$\omega : P \to \ZZ$ be an injective vertex-labeling of~$P$.
Following \cite[§7.19]{MR4621625}, a \Dfn{reverse $\boldsymbol{(P,\omega)}$-partition} is a map~$\sigma: P \to \NN$ satisfying
\begin{enumerate}
    \item if $a \prec_P b$, then $\sigma(a) \leq \sigma(b)$, and
    \item if $a \prec_P b$ and $\omega (a) > \omega(b)$,
        then $\sigma(a) < \sigma(b)$.
\end{enumerate}
Denote the set of reverse $(P,\omega)$-partitions by $\cA^r(P,\omega)$,
and define the quasisymmetric function
\[
    K_{P,\omega}
    = \sum_{\sigma \in \cA^r(P,\omega) } \boldsymbol{x}^\sigma \in \NN[\![\boldsymbol{x}]\!] = \NN[\![ x_1,x_2,x_3,\dots]\!] \,,
\]
where $\boldsymbol{x}^\sigma = \prod_{a\in P} x_{\sigma(a)}$.
Observe that $K_{P,\omega}$ only depends on the relative order of~$\omega$ on the cover relations of~$P$.
We exploit this property to prove \Cref{prop:Pomegapartitions}.

There is an alternative formula for $K_{P,\omega}$ in terms of linear extensions and the fundamental quasisymmetric functions
\[
    L_{\Des(\sigma)}
    = \sum x_{i_1} \cdots \ x_{i_n} \,,
\]
where the sum ranges over all sequences $1\leq i_1 \leq \cdots \leq i_n \leq n$ with $i_k < i_{k+1}$ if $k\in \Des(\sigma)$.
The quasisymmetric function $K_{P,\omega}$ then has the description
\begin{align}
\label{eq: K in fundamental basis}
    K_{P,\omega} = \sum_{\sigma \in \Lin(P) } L_{\Des \omega(\sigma)} \,,
\end{align}
see~\cite[Corollary~7.19.5]{MR4621625}.

\begin{proof}[Proof of \Cref{prop:Pomegapartitions}]
  Let $\omega,\omega' : P \rightarrow \ZZ$ be two injective vertex-labelings such that for each cover relation $a\prec_P b$, we have
  \[
    \omega(a) < \omega(b) \Longleftrightarrow \omega'(a) < \omega'(b)\,.
  \]
  Since $K_{P,\omega}$ only depends on the relative orders of $\omega$ on the cover relations, we obtain $K_{P,\omega} = K_{P,\omega'}$.
  It thus follows from~\eqref{eq: K in fundamental basis} that
  \[
    \sum_{\sigma \in \Lin(P) } L_{\Des\omega(\sigma)} = \sum_{\sigma \in \Lin(P) } L_{\Des\omega'(\sigma)}\,.
  \]
  Since $L_{\Des(\sigma)}$ only depends on the descent set of~$\sigma$, this implies \Cref{prop:Pomegapartitions}.
  More precisely, for any subsets $A,B \subseteq \Perm(n)$, the equalities of fundamental quasisymmetric function generating functions and of descent set generating functions are equivalent,
  \[
    \sum_{\sigma \in A } L_{\Des(\sigma)} = \sum_{\sigma \in B } L_{\Des(\sigma)}
    \quad \Longleftrightarrow \quad
    \sum_{\sigma \in A } \boldsymbol{t}^{\Des(\sigma)} = \sum_{\sigma \in B }  \boldsymbol{t}^{\Des(\sigma)}\,. \qedhere
  \]
\end{proof}

\printbibliography

\end{document}